\documentclass[a4paper, 12pt]{amsart}
\usepackage[utf8]{inputenc}
\usepackage[english]{babel}
\usepackage[T1]{fontenc}
\usepackage{amsmath, amssymb, amsfonts}
\usepackage[all]{xy}
\usepackage{enumerate}
\usepackage{graphicx}
\usepackage[hmargin=2.5cm, vmargin=2.5cm]{geometry}
\usepackage{rotate}
\usepackage{MnSymbol}
\usepackage{tikz}
\usepackage{comment}
\usepackage{hyperref}

\title[Representations of noncrossing quantum groups]{On the representation theory of some noncrossing partition quantum groups}
\author{Amaury Freslon}
\address{Laboratoire de Math\'ematiques d'Orsay, Univ. Paris-Sud, CNRS, Universit\'e Paris-Saclay, 91405 Orsay, France}
\email{amaury.freslon@math.u-psud.fr}
\date{}

\theoremstyle{plain}
\newtheorem{thm}{Theorem}[section]
\newtheorem{prop}[thm]{Proposition}
\newtheorem{cor}[thm]{Corollary}
\newtheorem{lem}[thm]{Lemma}

\theoremstyle{definition}
\newtheorem{de}[thm]{Definition}

\theoremstyle{remark}
\newtheorem{rem}[thm]{Remark}

\DeclareMathOperator{\Irr}{Irr}

\DeclareMathOperator{\Proj}{Proj_{\CC}}

\newcommand{\A}{\mathcal{A}}

\newcommand{\CC}{\mathcal{C}}

\newcommand{\DD}{\mathcal{D}}

\newcommand{\G}{\mathbb{G}}
\newcommand{\Gr}{\mathcal{G}}

\newcommand{\N}{\mathbb{N}}

\newcommand{\Z}{\mathbb{Z}}

\newcommand{\idpart}{|}

\begin{document}

\begin{abstract}
We compute the representation theory of two families of noncrossing partition quantum groups connected to amalgamated free products and free wreath products. This illustrates the efficiency of the methods developed in our previous joint work with M. Weber.
\end{abstract}

\maketitle

\section{Introduction}

Easy quantum groups are a class of compact quantum groups introduced by T. Banica and R. Speicher in \cite{banica2009liberation}. They are built from sets of partitions through a Tannaka-Kreain duality argument. In a joint work with M. Weber \cite{freslon2013representation}, we developed a general machinery to compute the representation theory of an easy quantum group out of its defining combinatorial data. However, our techniques are only completely effective when the partitions involved are \emph{noncrossing}. It turns out that the representation theory of all easy quantum groups associated to noncrossing partitions had already been computed at that time so that we could only show that our method allowed for a uniform and simple way to derive all these results.

Later on, we introduced in \cite{freslon2014partition} a generalization of easy quantum groups called \emph{partition quantum groups} to which the aforementioned results adapt straightforwardly. This potentially yields many new examples for which the tools of \cite{freslon2013representation} could be useful in describing the representation theory. In \cite{freslon2017noncrossing} we classified a family of partition quantum groups all of whose elements had never been studied before. The present paper will be concerned with computing the representation theory of some of these examples.

Our main point is to illustrate how effective the techniques of \cite{freslon2013representation} are. In particular, our first family of examples comes from amalgamated free products of compact quantum groups, for which there is no general method to study representation theory (unlike the case of classical discrete groups). Here we can very easily describe all irreducible representations together with the fusion rules. The second family is a generalization of free wreath products and once again, our method quite easily yields the result, in contrast with the technicality of the treatment of free wreath products in \cite{lemeux2013fusion}.

The main results of this paper were stated in an early version of \cite{freslon2017noncrossing}, together with sketches of proofs. We later decided to remove them and give full details in this independent article.

\section{Preliminaries}

In order to keep this work short, we will not recall all the definitions and results of the theory of partition quantum groups. We refer the reader to \cite{freslon2014partition} and the references therein for a detailed account. Our basic objects are partitions of finite sets, which we represent by drawing the elements of the set on two parallel rows and connecting them by lines if they belong to the same component of the partition. This graphical interpretation enables to concatenate partitions horizontally and vertically (if the number of points match) and rotate points from one line to another. A \emph{category of partitions} is a collection of partitions $\CC(k, l)$ on $k+l$ points for all integers $k, l$ which is globally stable under the above operations. It is moreover said to be \emph{coloured} if to each point is attached an element of a colour set $\A$. When a point of a partition is rotated from one row to another, its colour is changed into its inverse, which is defined through an involutive map $a\mapsto a^{-1}$ fixed once and for all on $\A$. To each partition $p$ we associate the word $w$ formed by the colours of the upper row (from left to right) and the word $w'$ formed by the colours of its lower row (again from left to right). We then write $p\in \CC(w, w')$.

The framework of easy quantum groups \cite{banica2009liberation} generalizes to this setting. In particular, to any category of coloured partitions $\CC$ and any integer $N$ is associated a compact quantum group $\G_{N}(\CC)$. It is generated by unitary representations $(u^{a})_{a\in \A}$ indexed by the colours and any irreducible representation is a subrepresentation of a tensor product
\begin{equation*}
u^{\otimes w} = u^{w_{1}}\otimes\cdots\otimes u^{w_{n}}.
\end{equation*}

Together with M. Weber, we developped in \cite{freslon2013representation} a method to compute the representation theory of $\G_{N}(\CC)$ by using only the combinatorial properties of $\CC$. The main tool is the so-called projective partitions :

\begin{de}
A partition $p$ is said to be \emph{projective} if $pp = p = p^{*}$. Moreover,
\begin{itemize}
\item A projective partition $p$ is said to be \emph{dominated} by $q$ if $qp = p$. Then, $pq = p$ and we write $p\preceq q$.
\item Two projective partitions $p$ and $q$ are said to be \emph{equivalent} in a category of partitions $\CC$ if there exists $r\in \CC$ such that $r^{*}r = p$ and $rr^{*} = q$. We then write $p\sim q$ or $p\sim_{\CC} q$ if we want to keep track of the category of partitions.
\end{itemize}
Note that if $p\sim q$, then $t(p) = t(q)$.
\end{de}

If $p$ and $q$ are equivalent in $\CC$, then they in fact both belong to $\CC$. Moreover, the equivalence is implemented by a unique partition denoted by $r_{p}^{q}$. As explained in \cite[Def 4.1]{freslon2013representation}, one can associate to any projective partition $p\in \CC(w, w)$ a subrepresentation $u_{p}$ of $u^{\otimes w}$. The study of these representations can be very complicated in general but in the noncrossing case (i.e. when the lines can be drawn so as not to cross) things become more tractable thanks to \cite[Thm 4.18 and Prop 4.22]{freslon2013representation} :

\begin{thm}
Let $\CC$ be a category of noncrossing partitions and let $N\geqslant 4$ be an integer. Then,
\begin{itemize}
\item $u_{p}$ is a non-zero irreducible representation for any projective partition $p$,
\item $u_{p}\sim u_{q}$ if and only if $p\sim q$,
\item $u^{\otimes w} = \displaystyle\bigoplus_{p\in \Proj(w)} u_{p}$ as a direct sum of representations. In particular, any irreducible representation is equivalent to $u_{p}$ for some projective partition $p$.
\end{itemize}
\end{thm}

To complete the description of the representation theory, we have to know how tensor products of irreducible representations split into sums of irreducible representations. This is often called the \emph{fusion rules} of the quantum group. Once again, in the noncrossing case a complete description can be obtained by combinatorial means. To explain this, we first need to introduce some specific partitions : we denote by $h_{\square}^{k}$ the projective partition in $NC(2k, 2k)$ where the $i$-th point in each row is connected to the $(2k-i+1)$-th point in the same row (i.e. an increasing inclusion of $k$ blocks of size $2$). If moreover we connect the points $1$, $k$, $1'$ and $k'$, we obtain another projective partition in $NC(2k, 2k)$ denoted by $h_{\boxvert}^{k}$. From this, we define binary operations on projective partitions, using $\idpart$ to denote the identity partition and adding suitable colours :
\begin{eqnarray*}
p\square^{k} q & = & (p_{u}^{*}\otimes q_{u}^{*})\left(\idpart^{\otimes t(p)-k}\otimes h_{\square}^{k}\otimes \idpart^{\otimes t(q)-k}\right)(p_{u}\otimes q_{u}) \\
p\boxvert^{k} q & = & (p_{u}^{*}\otimes q_{u}^{*})\left(\idpart^{\otimes t(p)-k}\otimes h_{\boxvert}^{k}\otimes \idpart^{\otimes t(q)-k}\right)(p_{u}\otimes q_{u})
\end{eqnarray*}
where $p = p_{u}^{*}p_{u}$ and $q = q_{u}^{*}q_{u}$ (such decompositions exist and are essentially unique, see \cite[Prop 2.9]{freslon2013representation}). We refer the reader to \cite[Subsec 4.5]{freslon2013representation} for more details.

\begin{thm}\label{thm:fusionrules}
Let $\CC$ be a category of noncrossing partitions and let $N\geqslant 4$ be an integer. Then, for any projective partitions $p, q\in \CC$,
\begin{equation*}
u_{p}\otimes u_{q} = u_{p\otimes q} \oplus\bigoplus_{k=1}^{\min(t(p), t(q))}u_{p\square^{k} q}\oplus u_{p\boxvert^{k} q}
\end{equation*}
with the convention that $u_{p\square^{k} q} = 0$ (resp. $u_{p\boxvert^{k} q} = 0$) if $p\square^{k} q$ (resp. $p\boxvert^{k} q$) is not in $\CC$.
\end{thm}

We end with some notations. The unique one-block partitions (i.e. with all its points connected) in $P(w, w')$ will be denoted by $\pi(w, w')$. If we consider instead the two-block partition in $P(w, w')$ where all the upper points are connected and all the lower points are connected, then we will denote it by $\beta(w, w')$.

\section{Amalgamated free products of free orthogonal quantum groups}\label{sec:biorthogonal}

Our first family of examples is linked to the free orthogonal quantum groups introduced by S. Wang in \cite{wang1995free}. In this section, $\A$ will consist in two colours $x$ and $y$ which are their own inverse. Recall from \cite{banica2009liberation} that if we denote by $\CC_{O^{+}}$ the category of all noncrossing pair partitions coloured by $x$, then $\G_{N}(\CC_{O^{+}}) = O_{N}^{+}$. Moreover, it has a projective version $PO_{N}^{+}$ generated by all the coefficients of $u^{x}\otimes u^{x}$. Consider now two copies of $\CC_{O^{+}}$, one coloured by $x$ and the other one by $y$ and form their free product amalgamated over their projective versions (see \cite{wang1995free} for the definition of amalgamated free products). This can be described as a partition quantum group using the partition $D_{xy}\in NC(2, 0)$ which is $\sqcup$ coloured with $x$ and $y$, see \cite[Lem 4.2]{freslon2017noncrossing}. We can in fact define a whole family of compact quantum groups in that way.

\begin{de}
For $\ell\geqslant 0$, we define a category of partitions $\CC_{O^{++}(\ell)}$ by
\begin{equation*}
\CC_{O^{++}(\ell)} = \left\langle \CC_{O^{+}}\ast\CC_{O^{+}}, D_{xy}^{*}D_{xy}, D_{xy}^{\otimes \ell}\right\rangle,
\end{equation*}
where by convention $D_{xy}^{\otimes 0} = \emptyset$ and $\langle\cdot\rangle$ denotes the category of partitions generated by a given set of partitions. Moreover, we write $O_{N}^{++}(\ell)$ for $\G_{N}(\CC_{O^{++}(\ell)})$.
\end{de}

The partition $D_{xy}^{*}D_{xy}$ is projective and the corresponding one-dimensional representation was explicitly described in \cite[Lem 4.3]{freslon2017noncrossing}. Let us denote by $\Gr(\G)$ the group of one-dimensional representations of a compact quantum group $\G$.

\begin{lem}\label{lem:cyclic1d}
For $1\leqslant k\leqslant N$, set
\begin{equation*}
s = \sum_{m=1}^{N}u_{km}^{x}u_{km}^{y}\in C(O_{N}^{++}).
\end{equation*}
Then, $s$ is a group-like element which does not depend on $k$. Moreover, it generates the group of group-like elements $\mathcal{G}(O_{N}^{++}(\ell))$ (which is therefore cyclic).
\end{lem}

As already mentioned,
\begin{equation*}
O_{N}^{++}(0) = O_{N}^{+}\underset{PO_{N}^{+}}{\ast}O_{N}^{+}
\end{equation*}
and $O_{N}^{++}(\ell)$ is obtained by quotienting this amalgamated free product by the relation $s^{\ell} = 1$. S. Wang gave in \cite{wang1995free} a description of all irreducible representations of a free product of quantum groups. However, there is no analogous result for amalgamated free products. That is why our computation is interesting : taking advantage of the partition structure, we can explicitly compute the representation theory of the amalgamated free product and its quotients. To do this, it will prove convenient to have an alternate description of $\CC_{O^{++}(\ell)}$. Let us first remark that any word on $\A = \{x, y\}$ can be seen as an element of the free product $\Z_{2}\ast\Z_{2}$ by sending $x$ to the generator of the first copy of $\Z_{2}$ and $y$ to the generator of the second one and denote by $\varphi$ the map sending a word to its image in $\Z_{2}\ast\Z_{2}$.

\begin{de}
For an integer $\ell$, let $\Gamma_{\ell}$ be the subgroup of $\Z_{2}\ast\Z_{2}$ generated by $\varphi(xy)^{\ell}$. A partition $p\in NC(w, w')$ is said to be \emph{$\ell$-admissible} if $\varphi(w)\varphi(w')^{-1}\in \Gamma_{\ell}$. The set of all $\ell$-admissible \emph{pair} partitions is denoted by $\DD_{\ell}$.
\end{de}

The point of this definition is that $\DD_{\ell}$ is a category of partitions and $D_{xy}^{*}D_{xy}, D_{xy}^{\otimes \ell}\in \DD_{\ell}$, hence $\CC_{O^{++}(\ell)}\subset \DD_{\ell}$. With this in hand, we are now ready to compute the fusion rules of $O_{N}^{++}(\ell)$. For convenience, we define the abstract labelling separately. By convention, we set $\Z_{0} = \Z$.

\begin{de}
For an integer $\ell$, let $W_{\ell}$ be the free monoid $\N\ast\Z_{\ell}$ and denote by $X$ and $\theta$ the generators of $\N$ and $\Z_{\ell}$ respectively. We endow this monoid with the unique antimultiplicative involution $w\mapsto \overline{w}$ induced by $\overline{X} = X$ and $\overline{\theta} = \theta^{-1}$. Eventually, we denote by $M_{\ell}$ the quotient of $W_{\ell}$ by the relation $X\theta = \overline{\theta}X$
\end{de}

\begin{thm}\label{thm:biorthogonalfusionrules}
The irreducible representations of $O_{N}^{++}(\ell)$ can be labelled, up to unitary equivalence, by the elements of $M_{\ell}$. Moreover, the fusion rules are given by the formula
\begin{equation*}
\theta^{k}X^{n}\otimes X^{n'}\theta^{k'} = \theta^{k+k'} \oplus \bigoplus_{i=0}^{\min(n, n')} \theta^{k}X^{n+n' - 2i}\theta^{k'}.
\end{equation*}
\end{thm}

\begin{proof}
Let $Z$ be the free monoid $\N\ast \N\ast \Z_{\ell}$ and denote by $X$, $Y$ and $\theta$ the generators. By noncrossingness, any projective partition in $\CC_{O^{++}(\ell)}$ is a tensor product of $\pi(x, x)$, $\pi(y, y)$ and projective partitions without through-blocks. The latter are, by Lemma \ref{lem:cyclic1d}, equivalent to tensor powers of $D_{xy}^{*}D_{xy}$ or of $D_{yx}^{*}D_{yx}$. Denoting by $X$ the equivalence class of $u_{\pi(x, x)}$, by $Y$ the equivalence class of $u_{\pi(y, y)}$ and by $\theta$ the one of $u_{D_{xy}^{*}D_{xy}}$, Lemma \ref{lem:cyclic1d} implies that the irreducible representations of $O_{N}^{++}(\ell)$ can be labelled (in a non-injective way) by $Z$. Moreover, rotating $D_{xy}^{*}D_{xy}$ yields the partition $\pi(x, y)\otimes D_{xy}$, which implements an equivalence between $\pi(x, x)\otimes (D_{xy}^{*}D_{xy})$ and $\pi(y, y)$. Hence, the relation $X\theta = Y$ holds and we do not need $Y$ to label the irreducible representations. Moreover
\begin{equation*}
\overline{\theta}X = \overline{Y} = Y = X\theta,
\end{equation*}
so that the irreducible representations of $O_{N}^{++}(\ell)$ can in fact be labelled by $M_{\ell}$. We now have a surjective map $M_{\ell}\rightarrow \Irr(O_{N}^{++}(\ell))$ and we claim that it is injective.

Note that thanks to the commutation relation $X\theta = \overline{\theta}X$, any element of $M_{\ell}$ can be written in the form $\theta^{k}X^{n}$ for $k\in \Z$ and $n\in \N$. Assume that there exist $n, n'\in \N$ and $k, k'\in \Z$ such that $\theta^{k}X^{n} = \theta^{k'}X^{n'}$ in $\Irr(O_{N}^{++}(\ell))$. This translates into the equivalence of projective partitions
\begin{equation*}
(D_{xy}^{*}D_{xy})^{\otimes k}\otimes \pi(x, x)^{\otimes n} \sim (D_{xy}^{*}D_{xy})^{\otimes k'}\otimes \pi(x, x)^{\otimes n'}.
\end{equation*}
By \cite[Lem 4.19]{freslon2013representation}, equivalent projective partitions have the same number of through-blocks. Since the number of through-blocks is the power of $\pi(x, x)$, we deduce that $n=n'$. But then, the equivalence reduces to $D_{xy}^{\otimes \vert k-k'\vert} = 1$. Because the image under $\varphi$ of this partition is $\varphi(xy)^{\vert k-k'\vert}$ and $\CC_{O^{++}(\ell)}\subset\DD_{\ell}$, it follows that $k = k' \mod \ell$, so that the two representations are already equivalent in $M_{\ell}$. We therefore conclude that $\Irr(O_{N}^{++}(\ell)) = M_{\ell}$.

For the tensor product, notice first that because of the commutation relations, any representation is equivalent to one as in the statement. Moreover,
\begin{equation*}
\pi(x, x)^{\otimes n}\square^{i}\pi(x, x)^{\otimes n'} \sim \pi(x, x)^{\otimes (n+n'-2i)}
\end{equation*}
while $\pi(x, x)^{\otimes n}\boxvert^{i}\pi(x, x)^{\otimes n'}\notin \CC_{O^{++}(\ell)}$.
\end{proof}

\begin{rem}\label{rem:amalgamateddecompositionorthogonal}
These fusion rules suggest that $O_{N}^{++}(\ell)$ may be decomposed as a product of $O_{N}^{+}$ and $\Z_{\ell}$ with some commutation relation. In fact, let $C(\G)$ be the quotient of the free product $C(O_{N}^{+})\ast C^{*}(\Z_{\ell})$ by the relations $t u_{ij} = u_{ij}t^{-1}$ for all $1\leqslant i, j\leqslant N$, where $t$ is the canonical generator of $C^{*}(\Z_{\ell})$ and $u$ is the fundamental representation of $O_{N}^{+}$. Then, the push-forward of the coproduct endows $\G$ with a compact quantum group structure which can be proven to be is isomorphic to $O_{N}^{++}$.
\end{rem}

A close look at the proof of Theorem \ref{thm:biorthogonalfusionrules} shows that it works for any category of partitions containing $\CC_{O^{++}(\ell)}$ and contained in $\DD_{\ell}$. The reason for that is, as one may guess, that these two categories are equal.

\begin{cor}
For any integer $\ell$, $\CC_{O^{++}(\ell)} = \DD_{\ell}$.
\end{cor}

\begin{proof}
By \cite[Thm 4.4]{freslon2017noncrossing}, the inclusion $\CC_{O^{++}(\ell)} \subset \DD_{\ell}$ implies that $\DD_{\ell} = \CC_{O^{++}(\ell')}$ for some integer $\ell'$. But then, it follows from the proof of Theorem \ref{thm:biorthogonalfusionrules} that
\begin{equation*}
\Z_{\ell} = \Gr(\G_{N}(\DD_{\ell})) = \Gr(O_{N}^{++}(\ell')) = \Z_{\ell'}
\end{equation*}
so that $\ell = \ell'$.
\end{proof}

\section{Free wreath products of pairs}

Our second family of examples is a generalization of the free wreath product construction of \cite{bichon2004free} called \emph{free wreath products of pairs} and introduced in \cite{freslon2017noncrossing}. Given a discrete group $\Gamma$ and an integer $N$, one can build a compact quantum group $\widehat{\Gamma}\wr_{\ast}S_{N}^{+}$ called the free wreath product of $\widehat{\Gamma}$ (this is the compact quantum group dual to $\Gamma$) by $S_{N}^{+}$. It was shown in \cite{lemeux2013fusion} that these objects are partition quantum groups. More precisely, let $S$ be a symmetric generating set of $\Gamma$ and consider a noncrossing partition coloured by $S$. We can associate to its upper and lower colouring elements of $\Gamma$ by simply multiplying the colours. Let us denote by $\varphi(w)$ the element thus obtained from a word $w$. If we denote by $\CC_{\Gamma, S}$ the set of all noncrossing partitions such that in each block with upper and lower colouring $w$ and $w'$ respectively, $\varphi(w) = \varphi(w')$, we get a category of partitions whose associated compact quantum group is by \cite[Thm 2.12]{lemeux2013fusion} exactly $\widehat{\Gamma}\wr_{\ast}S_{N}^{+}$. It turns out that this compact quantum group does not depend on $S$.

Let now $\Lambda\subset \Gamma$ be a fixed subgroup. Any element $\lambda\in \Lambda$ is by definition the image under $\varphi$ of a word $w_{\lambda}$ on $S$. Let us denote by $\CC_{\Gamma, \Lambda, S}$ the category of partitions generated by $\CC_{\Gamma, S}$ and $\beta(w_{\lambda}, w_{\lambda})$ for all $\lambda\in \Lambda$. The associated compact quantum group is denoted by $H_{N}^{++}(\Gamma, \Lambda)$ and called the \emph{free wreath product of the pair $(\Gamma, \Lambda)$ by $S_{N}^{+}$}. It is easily seen that this does not depend on $S$ or on the choice of the representatives $w_{\lambda}$. We will now compute the representation theory of this compact quantum group. To do this, it will be useful as in Section \ref{sec:biorthogonal} to embed $\CC_{\Gamma, \Lambda, S}$ into a category of partitions defined in a different way. This time however, the definition is subtler so that we first introduce some terminology :

\begin{de}
Let $p$ be a partition of $\{1, \cdots, k\}$. A subpartition $q$ (i.e.~a union of blocks of $p$) is said to be \emph{full} if it is a partition of $\{a, \cdots, a+b\}$ for some $1\leqslant a\leqslant a+b\leqslant k$.
\end{de}

For convenience, we give the definition of our auxiliary category of partitions as a Lemma.

\begin{lem}
Let $\DD_{\Gamma, \Lambda, S}$ be the set of all partitions $p\in NC^{S}(w, w')$ such that
\begin{itemize}
\item $\varphi(w) = \varphi(w')$,
\item For any full subpartition of $p$ with upper and lower colourings $v$ and $v'$ respectively, $\varphi(v)^{-1}\varphi(v')\in \Lambda$.
\end{itemize}
Then, this is a category of partitions and $\CC_{\Gamma, \Lambda, S}\subset \DD_{\Gamma, \Lambda, S}$.
\end{lem}

\begin{proof}
The first condition is clearly preserved under all category operations. As for the second one, only the stability under vertical concatenation is not obvious. So let $p, q\in \DD_{\Gamma, \Lambda, S}$ and consider a full subpartition $r$ of $qp$. If it lies on one line, it was already a full subpartition of $p$ or $q$ so that its colouring yields an element in $\Lambda$.

Otherwise, by definition $r$ consists in the first $a$ upper points of $p$ and the first $b$ lower points of $q$ or the last $a$ upper points of $p$ and the last $b$ lower points of $q$. Since the two cases are similar, we will only consider the first one. Let $r'$ be the smallest full subpartition of $p$ such that
\begin{itemize}
\item Its upper points are the first $a$ upper points of $p$,
\item It contains all the lower points of $p$ which are either connected to one of the first $a$ upper points of $p$ or correspond to upper points of $q$ which are connected to one of its first $b$ lower points.
\end{itemize}
Similarly, we consider the smallest subpartition $r''$ of $q$ such that
\begin{itemize}
\item Its lower points are the first $b$ lower points of $q$,
\item It contains all the upper points of $q$ which are either connected to one of the first $b$ lower points of $q$ or correspond to lower points of $p$ which are connected to one of its first $a$ upper points.
\end{itemize}
Let us denote by $a'$ (resp. $b'$) the number of lower (resp. upper) points of $r'$ (resp. $r''$). We claim that $a' = b'$. Indeed, assume that $a' > b'$ and consider the point $a'+1$ in $r''$. It does not correspond to an upper point of $q$ connected to one of the first $b$ lower points by definition, hence it is connected to one of the first $a$ upper points of $p$. But then the point $a'+1$ in the upper row of $q$ should be in $r''$, contradicting $a' > b'$. The impossibility of the other inequality follows in a similar way. Thus, the vertical concatenation $r''r'$ makes sense and must by definition equal $r$. Since both $r'$ and $r''$ satisfy the second condition in the statement, so does $r$.

As for the inclusion, simply notice that by definition $\CC_{\Gamma, S}\subset\DD_{\Gamma, \{e\}, S}\subset\DD_{\Gamma, \Lambda, S}$ and $\beta(w_{\lambda}, w_{\lambda})\in\DD_{\Gamma, \Lambda, S}$ for all $\lambda\in \Lambda$.
\end{proof}

With this characterization, we can start by finding all the one-dimensional representations of $H_{N}^{++}(\Gamma, \Lambda)$.

\begin{lem}\label{lem:onedwreath}
There is an isomorphism
\begin{equation*}
\Gr(H_{N}^{++}(\Gamma, \Lambda)) \simeq \Lambda.
\end{equation*}
\end{lem}

\begin{proof}
Let $p = q^{*}q$ be a non-through-block projective partition and let $w = g_{1}\cdots g_{n}$ be its upper colouring. Then,
\begin{equation*}
\beta(w, w) = \pi(w, w)p\pi(w, w)\in \CC_{\Gamma, \Lambda, S}
\end{equation*}
and since the upper points of $\beta(w, w)$ form a full subpartition, by the second defining condition of $\DD_{\Gamma, \Lambda, S}$ we have $\varphi(w) \in \Lambda$. Moreover, $\pi(w, w)p$ is an equivalence between $p$ and $\beta(w, w)$, hence the map
\begin{equation*}
\Phi :\Lambda\longrightarrow \Gr(H_{N}^{++}(\Gamma, \Lambda))
\end{equation*}
defined by $\Phi(\lambda) = \beta(w_{\lambda}, w_{\lambda})$ is a surjection. It is moreover compatible with the group multiplication and inverse, so that $\Gr(H_{N}^{++}(\Gamma, \Lambda))$ is a quotient of $\Lambda$.

Assume now that there exist $\lambda, \lambda'\in \Lambda$ such that $\beta_{\lambda} \sim \beta_{\lambda'}$. This equivalence can only be implemented by $\beta_{\lambda}\beta_{\lambda'}$ so that this partition must be in $\CC_{\Gamma, \Lambda, S}$. But by the first defining condition of $\DD_{\Gamma, \Lambda, S}$, this forces $\lambda = \lambda'$ and the result is proved.
\end{proof}

In order to describe higher-dimensional representations, let us introduce some notations : we denote by $F(\Gamma)$ the free monoid on $\Gamma$ and we consider the relation $\sim$ defined on it by
\begin{equation*}
w_{1}\dots (w_{i}.\lambda)w_{i+1}\dots w_{n}\sim w_{1}\dots w_{i}(\lambda.w_{i+1})\dots w_{n}
\end{equation*}
for any $1\leqslant i\leqslant n-1$ and $\lambda\in \Lambda$. The corresponding quotient will be denoted by $W(\Gamma, \Lambda)$.

\begin{thm}\label{thm:wreathfusionrules}
The one-dimensional representations of $H_{N}^{++}(\Gamma, \Lambda)$ can be labelled by $\Lambda$ and the other irreducible representations can be labelled by $W(\Gamma,  \Lambda)$.
\end{thm}

\begin{proof}
The result for one-dimensional representations was proved in Lemma \ref{lem:onedwreath}. For the other representations, we will split the proof into several steps.\\

\noindent\textbf{Step 1.} As a first step we claim that irreducible representations of $H_{N}^{++}(\Gamma, \Lambda)$ of dimension more than one can be labelled by $F(\Gamma)$. Consider a projective partition $p$ with at least one through-block. By noncrossingness it can be written as $p_{1}\otimes\cdots\otimes p_{n}$ where each $p_{i}$ is projective and $t(p_{i}) = 1$. If we denote by $w_{i}$ the upper colouring of $p_{i}$, then
\begin{equation*}
r = \left[\pi(w_{1}, w_{1})\otimes\cdots\otimes\pi(w_{n}, w_{n})\right]p
\end{equation*}
is an equivalence between $p$ and $\pi(w_{1}, w_{1})\otimes\cdots\otimes\pi(w_{n}, w_{n})$. This means that the map
\begin{equation*}
\Psi : \gamma_{1}\cdots\gamma_{n} \mapsto \pi(w_{\gamma_{1}}, w_{\gamma_{1}})\otimes\cdots\otimes\pi(w_{\gamma_{n}}, w_{\gamma_{n}})
\end{equation*}
induces a surjection from $F(\Gamma)$ onto the set of equivalence classes of irreducible representations of $H_{N}^{++}(\Gamma, \Lambda)$.\\

\noindent\textbf{Step 2.} The second step is to prove that they are in fact labelled by $W(\Gamma, \Lambda)$. Indeed, fix $\gamma, \gamma'\in\Gamma$ and $\lambda\in\Lambda$. Then,
\begin{align*}
r & = \left[\pi(w_{\gamma}.w_{\lambda}, w_{\gamma}.w_{\lambda})\otimes\pi(w_{\gamma'}, w_{\gamma'})\right]\left[\pi(w_{\gamma}, w_{\gamma})\otimes\beta(w_{\lambda}, w_{\lambda})\otimes\pi(w_{\gamma'}, w_{\gamma'})\right] \\
& \left[\pi(w_{\gamma}, w_{\gamma})\otimes\pi(w_{\lambda}.w_{\gamma'}, w_{\lambda}.w_{\gamma'})\right]
\end{align*}
yields an equivalence between the representations labelled by $(\gamma\lambda)\gamma'$ and $\gamma(\lambda\gamma')$, proving our claim. We will again denote by $\Psi$ the map obtained from $W(\Gamma, \Lambda)$ to the set of equivalence classes of higher-dimensional irreducible representations of $H_{N}^{++}(\Gamma, \Lambda)$.\\

\noindent\textbf{Step 3.} We now have to prove that there is no more relations, so let us consider $\gamma_{1}, \cdots, \gamma_{n}\in \Gamma$ and $\gamma'_{1}, \cdots, \gamma'_{k}\in \Gamma$ such that $\Psi(\gamma_{1}\cdots\gamma_{n}) = \Psi(\gamma'_{1}\cdots\gamma'_{k})$. By \cite[Thm 4.18 and Lem 4.19]{freslon2013representation}, if two projective partitions give equivalent representations, then they have the same number of through-blocks, so that the previous equality forces $k = n$. We will now prove by induction on $n$ that $\Psi(\gamma_{1}\cdots\gamma_{n}) = \Psi(\gamma_{1}'\cdots\gamma'_{n})$ implies that the two words are equal in $W(\Gamma, \Lambda)$. Let us start with the case $n = 2$. By assumption,
\begin{equation*}
r = \pi(w_{\gamma_{1}}, w_{\gamma'_{1}})\otimes\pi(w_{\gamma_{2}}, w_{\gamma'_{2}})\in \CC_{\Gamma, \Lambda, S}
\end{equation*}
and since $\pi(w_{\gamma_{1}}, w_{\gamma'_{1}})$ is a full subpartition of $r$, by the second defining condition of $\DD_{\Gamma, \Lambda, S}$ we have $\gamma_{1}^{-1}\gamma_{1}'\in \Lambda$. Let us denote by $\lambda$ this element. It follows that
\begin{align*}
[\pi(w_{\gamma'_{1}}, w_{\gamma'_{1}})\otimes\beta(w_{\lambda}, w_{\lambda})\otimes\pi(w_{\gamma'_{2}}, w_{\gamma'_{2}})]r & = \pi(w_{\gamma_{1}}, w_{\gamma_{1}})\otimes \beta(\emptyset, w_{\lambda})\otimes \pi(w_{\gamma_{2}}, w_{\gamma'_{2}}) \\
& \in \CC_{\Gamma, \Lambda, S}.
\end{align*}
Removing $\pi(w_{\gamma_{1}}, w_{\gamma_{1}})$ and capping yields $\pi(w_{\gamma_{2}}, w_{\lambda\gamma'_{2}})\in \CC_{\Gamma, \Lambda, S}$ so that by the first defining condition of $\DD_{\Gamma, \Lambda, S}$, $\gamma_{2} = \lambda\gamma'_{2}$. Summing up, $\gamma'_{1} = \gamma_{1}\lambda$ and $\gamma'_{2} = \lambda^{-1}\gamma_{2}$ and $\gamma_{1}\gamma_{2} = \gamma'_{1}\gamma'_{2}$ in $W(\Gamma, \Lambda)$.

Now for an arbitrary $n$, the reasoning is the same : we have
\begin{equation*}
r = \pi(w_{\gamma_{1}}, w_{\gamma_{1}'})\otimes\cdots\otimes \pi(w_{\gamma_{n}}, w_{\gamma'_{n}})\in \CC_{\Gamma, \Lambda, S}.
\end{equation*}
The second defining condition yields $\gamma_{1}^{-1}\gamma'_{1} \in \Lambda$ and by concatenating we get
\begin{equation*}
\pi(w_{\gamma_{2}}, w_{(\gamma_{1}^{-1}\gamma'_{1})\gamma_{2}'})\otimes\cdots\otimes \pi(w_{\gamma_{n}}, w_{\gamma'_{n}})\in \CC_{\Gamma, \Lambda, S},
\end{equation*}
from which we conclude by induction.\\
\end{proof}

As in Section \ref{sec:biorthogonal}, the larger category of partitions that we used in fact coincides with the original one.

\begin{cor}
For any inclusion $\Lambda\subset\Gamma$ and any finite symmetric generating set $S$, $\CC_{\Gamma, \Lambda, S} = \DD_{\Gamma, \Lambda, S}$.
\end{cor}

\begin{proof}
By \cite[Prop 3.16]{freslon2017noncrossing}, there exists a quotient map $\pi : \Gamma \to \Gamma'$ and a subgroup $\Lambda'\subset \Gamma'$ containing $\pi(\Lambda)$ such that $\DD_{\Gamma, \Lambda, S} = \CC_{\Gamma', \Lambda', S}$. By Lemma \ref{lem:onedwreath} and Theorem \ref{thm:wreathfusionrules}, $\Lambda = \Lambda'$ and $W(\Gamma, \Lambda) = W(\Gamma', \Lambda)$, yielding $\Gamma = \Gamma'$.
\end{proof}

We still have to describe the fusion rules of $H_{N}^{++}(\Gamma, \Lambda)$. For this purpose, we define two operations on $W(\Gamma, \Lambda)$. Given words $a = a_{1}\cdots a_{n}$, $b = b_{1}\cdots b_{n'}$ and $c$ such that $\varphi(c)\in \Lambda$, we set
\begin{align*}
a\underset{c}{\bullet}b & = a_{1}\cdots a_{n-1}(a_{n}.\varphi(c))b_{1}\cdots b_{n'} = a_{1}\cdots a_{n}(\varphi(c).b_{1})b_{2}\cdots b_{n'} \\
a\underset{c}{\ast}b & = a_{1}\cdots a_{n-1}(a_{n}.\varphi(c).b_{1})b_{2}\cdots b_{n'}.
\end{align*}
We can now compute the fusion rules.

\begin{prop}
The fusion rules of $H_{N}^{++}(\Gamma, \Lambda)$ are given by
\begin{align*}
(w_{1}\dots w_{n})\otimes (w_{1}'\dots w_{n'}') & = \sum_{w = az, w' = z'b, \varphi(zz')\in \Lambda} a\underset{zz'}{\bullet}b + a\underset{zz'}{\ast}b \\
\lambda\otimes (w_{1}\dots w_{n}) & = (\lambda.w_{1})w_{2}\dots w_{n}
\end{align*}
\end{prop}

\begin{proof}
The second equality follows directly from the definitions. As for the first one, let us note that for $k\leqslant \min(n, n')$, concatenating $w\square^{k}w'$ or $w\boxvert^{k}w'$ by
\begin{equation*}
\pi(w_{n-k+1}\cdots w_{n}w'_{1}\cdots w'_{k}, w_{n-k+1}\cdots w_{n}w'_{1}\cdots w'_{k})
\end{equation*}
in the middle we obtain the blocks of $\beta(w_{n-k+1}\cdots w_{n}w'_{1}\cdots w'_{k}, w_{n-k+1}\cdots w_{n}w'_{1}\cdots w'_{k})$. This means that the original partition is in $\CC_{\Gamma, \Lambda, S}$ only if the latter is, which in turn means that $\varphi(w_{n-k+1}\cdots w_{n}w'_{1}\cdots w'_{k})\in \Lambda$. In other words, the tensor product splits as a sum over all the decompositions $w = az$, $w' = z'b$ with $\varphi(zz')\in \Lambda$.

Considering now such a decomposition with $a$ of length $k$, then as explained in the previous paragraph $q\square^{k}w'$ is equivalent to $a\underset{zz'}{\ast b}$ while concatenating by
\begin{equation*}
\pi(w_{n-k}\cdots w_{n}w'_{1}\cdots, w'_{k+1}, w_{n-k}\cdots w_{n}w'_{1}\cdots, w'_{k+1})
\end{equation*}
shows that $w\boxvert^{k}w'$ is equivalent to $a\underset{zz'}{\bullet}b$.
\end{proof}

\bibliographystyle{amsplain}
\bibliography{../../../quantum}

\end{document}